\documentclass{article}
\usepackage{amsmath,amsthm,amsfonts,amssymb}
\usepackage[all]{xypic}

\theoremstyle{plain}

\newtheorem{thm}{Theorem}

\theoremstyle{definition}

\newcommand{\bea}{\begin{eqnarray}}
\newcommand{\eea}{\end{eqnarray}}

\newcommand{\nn}{\nonumber}

\begin{document}

\title{Partial Chromatic Polynomials and Diagonally Distinct Sudoku Squares} 
\author{F\"{u}sun Akman\\ {\small {\it Illinois State University, Normal, IL}}\\ {\small akmanf@ilstu.edu}}

\maketitle 

\section{Introduction}

This paper is based on a talk I gave at Illinois State University on April 10, 2008, and contains two proofs. The first one is of a statement about completions of partial $\lambda$-colorings of a graph in a very interesting article by Herzberg and Murty~\cite{HM}, namely, the fact that the number of possible completions is a polynomial in $\lambda$ (which we will call the {\it partial chromatic polynomial}). Two elegant proofs of this statement, one with M\"{o}bius inversion and the other by induction, are already given in~\cite{HM}: both proofs use the concept of contraction. Our alternative proof mimics the construction of the classical chromatic polynomial instead. The second proof in this paper shows that there exist $n^2\times n^2$ Sudoku squares with distinct entries in both diagonals (in addition to the rows, columns, and $n\times n$ sub-grids) for all $n$. I would like to thank Walter ``Wal'' Wallis for pointing out (days after I posted the proof and gave the talk) that there exists an earlier and very similar proof of the existence of such squares, due to A.D.~Keedwell~\cite{Ke}: I was unaware of \cite{Ke} at the time. I would also like to thank Papa Sissokho for correcting my terminology and making the first proof more palatable.

\section{Partial chromatic polynomial}

Sudoku puzzles are, in a discrete mathematician's world, partially colored graphs. Questions about the minimal number of clues for unique solutions etc.~boil down to questions about partial colorings of the ``Sudoku graph''. This graph consists of $n^4$ vertices, corresponding to the squares of an $n^2\times n^2$ Sudoku grid, such that any two distinct vertices in the same row, column, or sub-grid are joined by an edge. A completed Sudoku puzzle is then a proper coloring of the Sudoku graph with $n^2$ colors. 

\begin{thm}\cite{HM}
Let $G$ be a graph with $n$ vertices, and $C$ be a partial proper coloring of $t$ vertices of $G$ using exactly $\lambda_0$ colors. Define $p_{G,C}(\lambda)$ to be the number of ways $C$ can be completed to a proper $\lambda$-coloring of $G$. Then for $\lambda\geq\lambda_0$, the expression $p_{G,C}(\lambda)$ is a monic polynomial in $\lambda$ of degree $n-t$.
\end{thm}

\begin{proof}
Let $C$ be a partial proper coloring of $t$ vertices of $G$ with exactly $\lambda_0$ colors. Call a proper coloring $C'$ of $G$ ``consistent with $C$'' if the vertices colored under $C$ keep their colors under $C'$. Also call a proper coloring $C'$ ``generic'' if it is simply a partitioning of the vertices of $G$ into independent sets (more precisely, a generic coloring is an equivalence class of colorings with the same independent sets). Now let $C'$ be any generic proper coloring of $G$ with exactly $\lambda_0$ independent sets. If $C'$ is consistent with $C$, then there is only 1 way the colors of $C'$ can be specified; the larger independent sets  in $C'$ have to retain the colors of the smaller ones in $C$. Next, if a generic $C'$ is to be consistent with $C$ and have $\lambda_0+1$ independent sets, then there are $(\lambda-\lambda_0)$ ways of specifying the colors of $C'$: for the $\lambda_0$ sets that extend those in $C$, we have no choice but to respect the colors dictated by $C$. On the other hand, the extra independent set does not intersect $C$, so we can use any of the remaining $(\lambda-\lambda_0)$ colors. We continue the argument for all generic proper colorings with exactly $\lambda_0+r$ independent sets, where $0\leq r\leq n-t$. In short, we have 
\[ p_{\small G,C}(\lambda)=\sum_{r=0}^{n-t}m_r(G,C)(\lambda-\lambda_0)\cdots (\lambda-\lambda_0-r+1).\]
Here $m_r(G,C)$ is the number of generic proper colorings $C'$ of $G$ that are consistent with $C$ and have exactly $\lambda_0+r$ independent sets, and $(\lambda-\lambda_0)\cdots (\lambda-\lambda_0-r+1)$ is the number of ways the colors of such $C'$ can be specified. The $r$-th term of the sum is a polynomial of degree $r$, and the $(n-t)$-th term is monic, because there is only one generic $C'$ that adds $n-t$ independent sets to $C$. Namely, each vertex outside $C$ is a set by itself.
\end{proof}

\section{Diagonally distinct Sudoku squares}

The existence of $n^2\times n^2$ Sudoku squares for any positive integer $n$ is a well-known fact (see \cite{HM} for a proof). We will show that it is moreover possible to construct $n^2\times n^2$ Sudoku squares with distinct entries on each of the two diagonals for any $n$. A similar proof was given earlier, and unknown to the author at the time of e-publication of the first version of this paper, by Keedwell~\cite{Ke}. Michalowski et al.~\cite{MKC} and Bailey et al.~\cite{B} give  some motivating real-life examples for variations of Sudoku puzzles and other gerechte designs. 
\begin{thm}
There exist $n^2\times n^2$ Sudoku squares with distinct entries in the two diagonals, in addition to distinct entries in each row, column, and $n\times n$ sub-grid.
\end{thm}

\begin{proof}
Notation: the $(i,j)$-{\it block} will be the $n\times n$ sub-grid placed according to matrix-entry enumeration convention inside the full grid ($i$th from the top and $j$th from the left). We will also enumerate entries in any block by the row and column numbers in the block; thus, the $(r,c)$-entry of the $(i,j)$-block will be the $(r+(i-1)n,c+(j-1)n)$-entry of the complete grid. When writing indices, we will always choose the least positive residue modulo~$n$ (denoted by $[x]$ for any integer~$x$). As a result, all variables $i,j,r,c,[x]$ will have values in the set $\{ 1,\dots,n\}$.

Let the symbols $a(r,c)$, with $1\leq r,c\leq n$, denote the $n^2$ integers from 1 to $n^2$ in some order. We place these distinct integers in the upper left $n\times n$ block of the grid, now called the $(1,1)$-block, such that $a(r,c)$ is in row~$r$ and column~$c$:
\[    \begin{array}{||c|c|c||}\hline\hline
{a(1,1)} & {a(1,2)} & {a(1,3)} \\ \hline
{a(2,1)} & {a(2,2)} & {a(2,3)}\\ \hline 
 {a(3,1)} & {a(3,2)} & {a(3,3)}\\
\hline\hline
\end{array}
=                 
\begin{array}{||c|c|c||}\hline\hline
{1} & {2} & {3} \\ \hline
{4} & {5} & {6}\\ \hline 
 {7} & {8} & {9}\\
\hline\hline
\end{array}\]
In order to create the $(1,2)$-block, we simply move the rows of the $(1,1)$-block up in a cyclic fashion:
\[    \begin{array}{||c|c|c||}\hline\hline
{a(2,1)} & {a(2,2)} & {a(2,3)}\\ \hline
{a(3,1)} & {a(3,2)} & {a(3,3)}\\ \hline
{a(1,1)} & {a(1,2)} & {a(1,3)} \\ \hline
 
\hline
\end{array}
=                 
\begin{array}{||c|c|c||}\hline\hline
{4} & {5} & {6}\\ \hline
{7} & {8} & {9}\\ \hline
{1} & {2} & {3} \\ \hline
\hline
\end{array}\]
We continue this permutation of rows inside each new block until we finish the first row of blocks. As for the $(2,1)$-block, we advance the rows inside the $(1,1)$-block one step down cyclically, and also move the entries in each row (inside the block) one step backward:
\[    \begin{array}{||c|c|c||}\hline\hline
{a(3,2)}&{a(3,3)} & {a(3,1)}    \\ \hline
{a(1,2)}&{a(1,3)} & {a(1,1)}   \\ \hline
 {a(2,2)}&{a(2,3)} & {a(2,1)}   \\ \hline 
\hline
\end{array}
=                 
\begin{array}{||c|c|c||}\hline\hline
   {8}& {9} &{7}\\ \hline
  {2}&{3} &{1}   \\ \hline
{5}&{6}& {4}    \\ \hline 
\hline
\end{array}\]
We complete the second row of blocks similar to the first, only by permuting whole rows in the $(2,1)$-block upward, without making any changes to the rows internally, and repeat the procedure until all rows of blocks are exhausted. The $4\times 4$, $9\times 9$, and $16\times 16$ Sudoku squares with distinct diagonal entries constructed by this method are given below:
\[ \begin{array}{||c|c||c|c||}\hline\hline
1 & 2 & 3 &4 \\ \hline
3 & 4 & 1 & 2 \\ \hline \hline
4 & 3 & 2 & 1 \\ \hline
2 & 1 & 4 & 3 \\
\hline\hline
\end{array}\]

\[ \begin{array}{||c|c|c||c|c|c||c|c|c||}\hline\hline
 1 & 2 &3 & 4 & 5 & 6 & 7 & 8 & 9\\ \hline
4 & 5 & 6 & 7 & 8 & 9 & 1 & 2 &3 \\ \hline 
 7 & 8 & 9 & 1 & 2 &3 & 4 & 5 & 6   \\ \hline
\hline
8 & 9 & 7 & 2 &3 & 1 & 5  & 6  &4\\ \hline
2 &3 & 1 & 5 & 6 & 4 & 8 & 9& 7\\ \hline
5 & 6 & 4 & 8 & 9 & 7 & 2 & 3 &1\\ \hline 
\hline
6 & 4 & 5 & 9 & 7 & 8 & 3 & 1 & 2 \\ \hline 
 9 &7 & 8 & 3 & 1 &2 & 6 & 4  & 5  \\ \hline
 3 & 1 &2 & 6 & 4 & 5 & 9 & 7 & 8\\ \hline
\hline
\end{array}\]

\[ \begin{array}{||c|c|c|c||c|c|c|c||c|c|c|c||c|c|c|c||}\hline\hline
 1 & 2 & 3 & 4 & 5 & 6 & 7 & 8 & 9 & \!10\!& \!11\! & \!12\! &\! 13\!& \!14\! & \!15\! & \!16\! \\ \hline
 5 & 6 & 7 & 8  & 9 & \!10\! & \!11\! & \!12\! & \! 13\! & \!14\! & \!15\! & \!16\!& 1 & 2 & 3 & 4 \\ \hline 
9 & \!10\! & \!11\! & \!12\! & \! 13\! & \!14\! & \!15\! & \!16\! & 1 & 2 & 3 & 4 & 5 & 6 & 7 & 8 \\ \hline
\! 13\! & \!14\! & \!15\! & \!16\! & 1 & 2 & 3 & 4 & 5 & 6 & 7 & 8 & 9 & \!10\! & \!11\! & \!12\!\\ \hline
\hline
 \!14\! & \!15\! & \!16\! & \! 13\!  & 2 & 3 & 4 & 1 & 6 & 7 & 8 & 5 & \!10\!& \!11\! & \!12\! & 9 \\ \hline
2 & 3 & 4 & 1 &  6 & 7 & 8 & 5 & \!10\!& \!11\! & \!12\! & 9 & \!14\! & \!15\! & \!16\! & \! 13\! \\ \hline
6 & 7 & 8 & 5  & \!10\!& \!11\! & \!12\! & 9 & \!14\! & \!15\! & \!16\! & \! 13\! & 2 & 3 & 4 & 1 \\ \hline 
\!10\!& \!11\! & \!12\! & 9 & \!14\! & \!15\! & \!16\! & \! 13\! & 2 & 3 & 4 & 1 & 6 & 7 & 8 & 5 \\ \hline
\hline
\!11\! & \!12\! & 9 & \!10\!&  \!15\! & \!16\! & \! 13\! & \!14\! & 3 & 4 & 1 & 2 & 7  & 8 & 5 & 6 \\ \hline
\!15\! & \!16\! & \! 13\! & \!14\!  & 3 & 4 & 1 & 2  &  7  & 8 & 5 & 6 & \!11\! & \!12\! & 9 & \!10\! \\ \hline 
3 & 4 & 1 & 2 &   7  & 8 & 5 & 6 & \!11\! & \!12\! & 9 & \!10\! & \!15\! & \!16\! & \! 13\! & \!14\! \\ \hline
7  & 8 & 5 & 6 & \!11\! & \!12\! & 9 & \!10\! &\!15\! & \!16\! & \! 13\! & \!14\!& 3 & 4 & 1 & 2 \\ \hline
\hline
8 & 5 & 6 & 7 & \!12\! & 9 & \!10\!& \!11\!  & \!16\!& \! 13\!& \!14\! & \!15\!& 4 &1 & 2 & 3\\ \hline 
\!12\! & 9 & \!10\!& \!11\!  & \!16\!& \! 13\!& \!14\! & \!15\! & 4 &1 & 2 & 3 &  8 & 5 & 6 & 7  \\ \hline
\!16\!& \! 13\!& \!14\! & \!15\! & 4 &1 & 2 & 3 &  8 & 5 & 6 & 7 & \!12\! & 9 & \!10\!& \!11\! \\ \hline
4 &1 & 2 & 3 &  8 & 5 & 6 & 7 &\!12\! & 9 & \!10\!& \!11\! & \!16\!& \! 13\!& \!14\! & \!15\!  \\ \hline
\hline
\end{array}\]

We now present the full proof of existence. Let us place the integer 
\[ a\, (\, [\, r-(i-1)+(j-1)\, ]\, ,[\, c+(i-1)\, ]\,)=a\, (\, [\, r-i+j]\, ,[\, c+i-1]\, )\]
in the $(r,c)$-entry of the $(i,j)$-block, and use the prime notation to distinguish another entry. Distinct entries in the same row of the full grid (where $i=i'$ and $r=r'$, but $j\neq j'$ or $c\neq c'$) are not equal: if they were, then we would have
\bea &&[r-i+j]=[r-i+j'] \;\;\mbox{and}\;\; [c+i-1]= [c'+i-1]     \nn\\
&\Rightarrow& j=j' \;\;\mbox{and}\;\; c=c'.
\nn\eea
Similarly, distinct entries in the same column of the full matrix (where $j=j'$ and $c=c'$, but $i\neq i'$ or $r\neq r'$) cannot be equal:
\bea &&[r-i+j]=[r'-i'+j] \;\;\mbox{and}\;\; [c+i-1]= [c+i'-1]     \nn\\
&\Rightarrow& i=i'\;\;\mbox{and}\;\; r=r'.
\nn\eea
Two distinct entries in the same block (where $i=i'$ and $j=j'$, but $r\neq r'$ or $c\neq c'$) are not equal:
\bea &&[r-i+j]=[r'-i+j] \;\;\mbox{and}\;\; [c+i-1]= [c'+i-1]     \nn\\
&\Rightarrow& r=r'\;\;\mbox{and}\;\; c=c' .
\nn\eea
Two distinct entries on the main diagonal (where $i=j$, $i'=j'$, $r=c$, and $r'=c'$, but $i\neq i'$ or $r\neq r'$) are not equal:
\bea &&[r-i+i]=[r'-i'+i'] \;\;\mbox{and}\;\; [r+i-1]= [r'+i'-1]     \nn\\
&\Rightarrow&  r=r' \;\;\mbox{and}\;\; i=i'.
\nn\eea
Finally, two distinct entries on the secondary diagonal (where $i+j=i'+j'=n+1$, $r+c=r'+c'=n+1$, but $i\neq i'$ or $r\neq r'$) are not equal:
\bea &&[r-i+(n+1)-i]=[r'-i'+(n+1)-i'] \nn\\
&&\mbox{and}\;\; [(n+1)-r+i-1]= [(n+1)-r'+i'-1]     \nn\\
&\Rightarrow& [r-2i]=[r'-2i'] \;\;\mbox{and}\;\; [-r+i]=[-r'+i']     \nn\\
&\Rightarrow&  r=r' \;\;\mbox{and}\;\; i=i' .
\nn\eea
\end{proof}

Calculations of the symmetries, the number of essentially different squares, the minimum number of entries in a puzzle to assure a unique solution, the asymptotic values of related  expressions, and the partial or full chromatic polynomials for Sudoku graphs of rank~$n$, are mentioned in~\cite{HM} in relation to standard Sudoku squares. Similar calculations would certainly be interesting for the diagonally distinct $n^2\times n^2$ Sudoku squares.


\end{document}